\documentclass{article}
\usepackage{mathrsfs}
\usepackage{graphicx}
\usepackage{amsmath}
\usepackage{amssymb}
 \usepackage{txfonts}
\newtheorem{thm}{Theorem}[section]

\newtheorem{lem}[thm]{Lemma}
\newtheorem{prop}[thm]{Proposition}

\begin{document}

\begin{center}

{\Large \bf   On the number of non-cyclic subgroups of   finite $p$-groups \footnote{* Corresponding author: W. Meng ( mlwhappyhappy@163.com).W. Meng is supported by  Guangxi Natural Science
  Foundation Program (2025JJA110017), National Natural Science Foundation of China (12161021), Center for Applied Mathematics of Guangxi(GUET) and  Guangxi Colleges and Universities Key Laboratory of Data Analysis and Computation. Jia Liu is supported by Innovation Project of GUET Graduate Education(2025YCXS124).
}
}

\end{center}

\vskip0.5cm

\begin{center}

Jia Liu$^a$, Li Ma$^b$, Wei Meng$^a*$

 a. School of Mathematics  and Computing Science, Guilin University of Electronic Technology,
   Guilin, Guangxi, 541002, P.R. China.

b. School of Mathematics and Statistics, Qujing Normal University, Qujing, Yunnan, 655000, P.R. China.

  E-mails: 3542414643@qq.com, 821962782@qq.com,  mlwhappyhappy@163.com,
\end{center}

\begin{abstract}
Let $G$ be a finite $p$-group and $\delta(G)$ denote the number of all non-cyclic subgroups of $G$. In this paper, an upper bound  for $\delta(G)$ is obtained. Furthermore,  we  prove that $\delta(G)\leq \delta(M_p(1, 1, 1) \times C_{p}^{n-3})$ (if $p=2$, then $\delta(G)\leq \delta(D_8\times C_{2}^{n-3})$),  for any  non-elementary  abelian $p$-group  $G$ of order $p^n$.
\end{abstract}



\section{Introduction}

Throughout, all groups considered are  finite $p$-groups where $p$ is a prime number. The problem of counting subgroups is an important aspect of the study of finite $p$-groups. Let $G$ be a finite group of order $p^n$ and $s_m(G)$ represent the number of subgroups of order $p^m$ in $G$.   It is well known that $G$ is a cyclic group if $s_m(G) = 1$ for all $1\leq m\leq n$ (see \cite {BE2}).  Kulakoff \cite {KU} showed that $s_m(G)\equiv 1$ or $1+p$ (mod $p^2$) for $p>2$. Hua and Tuan   studied the possible cases of the number of subgroups $s_m(G)$ of the $p$-group $G $ modulo $p^3$ and other counting problems, and obtained many remarkable counting theorems (See \cite {HU1, HU2, TU1}). Later,   many scholars    investigated     $s_m(G)$ modulo $p^3$, such as Djubjuk \cite {DJ}, Berkovich \cite {BE1}, Zhang and Qu \cite {QU1, ZH1, ZH2}.

On the other hand, Fan \cite {FA} studied an upper bound for $s_m(G)$  in finite $p$-groups. He proved that $s_1(G)\leq  \left [\begin{array}{c}
n \\
1
\end{array}
\right ]_p$, with   equality if and only if $\exp(G)=p$, and $s_m(G)\leq \left [\begin{array}{c}
n \\
m
\end{array}
\right ]_p$ for $2\leq m<n$, with  equality if and only if $G$ is an elementary abelian $p$-group.
Fan's results shows that the $p$-group with the largest number of subgroups of each order is an elementary abelian $p$-group. A natural question is: what are the $p$-groups with the largest number of subgroups of each order like apart from elementary abelian $p$-groups?

  In 2013, Qu \cite {QU2} answered the above question for $p>2$. He  proved that if $G$ is a group of order $p^n$ and $N$ is a normal subgroup of order $p$ of $G$, then $s_k(G) \leq s_k(G/N \times C_p)$ for all $ 1 \leq k \leq n$.  Furthermore,   let
$\widetilde{G} = M_p(1, 1, 1) \times C_p^{n-3}$, where $M_p(1,1,1)$ is a non-abelian group of order $p^3$ with exponent $p$. Qu \cite {QU2} also showed that if $G$ is non-elementary abelian $p$-group, then $s_k(G) \leq s_k(\widetilde{G})$ for all $1 \leq k \leq n$.
 Let $s(G)$ denote the number of all subgroups of $G$. Then the results of Qu showed that the inequality  $s(G)\leq s(M_p(1, 1, 1) \times C_p^{n-3})$   holds for all non-abelian $p$-groups.  As a continuation of the above investigations, T\u{a}rn\u{a}ucean \cite {TA} proved that if $G$ is a non-elementary abelian $2$-group of order $2^n$, then
$s(G) \leq  s(D_8 \times C_{2}^{n-3})$. Dually, Aivazidis and M\"{u}ller \cite {AI} obtained the lower bound of $s(G)$ for finite non-cyclic $p$-groups.  Meng and Lu \cite {ME2} generalized the results of Aivazidis and MM\"{u}ller on all finite non-cyclic
nilpotent groups with specific order.

Recently, Lazorec,    Shen and  T\u{a}rn\u{a}uceanu \cite {LA} investigated the second minimum(or maximum) value of the number of cyclic subgroups of finite $p$-groups. Let $c(G)$ denote the number of all cyclic subgroups of $G$ and $c_m(G)$ denote the number of cyclic-subgroups of $G$ of order $p^m$. They \cite {LA}  proved that
$ c(G) \leq \frac { p^n + p^2 - p - 1 + (p - 1)^2 c_1(G)}{p^2 - p}$.  In \cite {ME1}, Meng and Lu also determined the lower bound of $c(G)$ for finite non-cyclic nilpotent groups.

\vskip 0.3cm

In the light of above investigations, in this paper,  we  consider the number of   non-cyclic subgroups of finite $p$-groups. For convenience, write  $\delta_m(G)$ for the number of non-cyclic subgroups of $G$ of order $p^m$ and  $\delta(G)$ for  the number of all non-cyclic subgroups of $G$. Our first goal is to give the upper bound of $\delta(G)$ for finite $p$-groups $G$. We prove the following results.

 \begin{thm}
Let $G$ be a group of order $p^n$. Then
\begin{center}
 $\delta(G)\leq \sum\limits_{k=2}^n\left [\begin{array}{c}
                                                                              n \\
                                                                              k
                                                                            \end{array}\right ]_p$,
 \end{center}
and the equality holds if and only if $G$ is an elementary abelian $p$-group of order $p^n$.
\end{thm}

 Furthermore, we investigate an upper bound for $\delta(G)$ for finite non-elementary abelian $p$-groups $G$. Similar to the results of Qu \cite {QU2}, we have the following results.

 \begin{thm}
Assume $G$ is a group of order $ p^n$, where $p$ is an odd prime.
 If $G$ is not elementary abelian, then
\begin{center}
$\delta(G)\leq \delta(M_p(1, 1, 1) \times C_{p}^{n-3})$.
\end{center}
\end{thm}

Finally, for finite non-elementary abelian $2$-groups, we obtain the following result.

\begin{thm}
Assume $G$ is a group of order $2^n$.
  If $G$ is not elementary abelian, then
\begin{center}
$\delta(G)\leq \delta(D_8\times C_{2}^{n-3})$.
\end{center}
\end{thm}

 All unexplained notation and terminologies
    are standard and can be found in \cite {BE2, SU}. In addition, $C_n$  denotes the cyclic group of order $n$;     $A\times B$ means a direct product of $A$ and $B$; $T^d$ denotes the direct product of  $d$ groups $T$; $A* B$ denotes the (amalgamated)
central product of the groups $A$ and $B$ having isomorphic centres, and $X^{*r}$
denotes the central product of $r$ copies of the group $X$.

\section{Preliminaries}
In this section, we collect some results which will be used in
the proof of the main results. For convenience,
 we list some known results about the number
$\left [\begin{array}{c}
n \\
m
\end{array}
\right ]_p$
 of subgroups of order $p^m$ of elementary abelian p-groups of order $p^n$.
 \begin{lem} \cite {BE2}
 $(1)$\[
\left [
\begin{array}{c}
n \\
m
\end{array}
\right ]_p =
\left\{
\begin{array}{ll}
\frac{(p^n - 1)(p^{n - 1} - 1)\cdots(p^{n - m + 1} - 1)}{(p^m - 1)(p^{m - 1} - 1)\cdots (p - 1)}, & n > m > 0 \\
1, & m < 0
\end{array}
\right.
\]

$(2)$ If $n \geq m$, then
$\left [\begin{array}{c}
n \\
m
\end{array}
\right ]_p$=
$\left [\begin{array}{c}
n \\
n-m
\end{array}
\right ]_p$

$(3)$ For $n$ and $m$,
$\left [\begin{array}{c}
n+1 \\
m
\end{array}
\right ]_p$=
$\left [\begin{array}{c}
n \\
m
\end{array}
\right ]_p + p^{n-m+1}\left [\begin{array}{c}
n \\
m-1
\end{array}
\right ]_p$
 \end{lem}
 \begin{lem}  \cite{BE2}
 Assume $\vert G \vert = p^n$, $1 < m < n$. If $s_m(G)=1$, then $G$ is cyclic.
 \end{lem}

 \begin{lem} \cite {BE2, FA}
 Assume $G$ is a group of order $p^n$.  Then

 (1) $s_1(G) \leq$
 $\left [\begin{array}{c}
n \\
1
\end{array}
\right ]_p$,
 in particular, the equality holds if and only if $\exp(G) = p$.

 (2)  If $1 < k < n$, then $s_k(G) \leq$
$\left [\begin{array}{c}
n \\
k
\end{array}
\right ]_p$,
the equality holds if and only if $G$ is
an elementary abelian $p$-group.
\end{lem}

\begin{lem} \cite {QU2}
Assume $G$ is a group of order $ p^n$, where $p$ is an odd prime, and
$\widetilde{G}= M_p(1, 1, 1) \times C_{p}^{n-3}$. If $G$ is non-elementary abelian, then $\forall 1\leq k \leq  n$,
$s_k(G) \leq s_k(\widetilde{G})$. In particular,  if $2 \leq k \leq n-2$, then $s_k(G) < s_k( \widetilde{G})$.
\end{lem}

\begin{lem} \cite {TA}
Let $G$ be a finite non-elementary abelian $2$-group of order $2^n$($n\geq 3$). Then
$s_k(G)\leq s_k(D_8\times  C_{2}^{n-3})$, $\forall 0\leq k\leq n$.
\end{lem}

Recall that  a finite $2$-group is called:

- extraspecial if $Z(G) = G' = \Phi(G)$ has order $2$;

- almost extraspecial if $G' = \Phi(G)$ has order $2$ and $Z(G)\leq C_4$;

- generalized extraspecial if $G' = \Phi(G)$ has order $2$ and $G'\leq Z(G)$.

The structure of these groups is well-known as follows.

\begin{lem} \cite {BO, ST}
Let $G$ be a finite $2$-group.

(1) If $G$ is extraspecial, then $|G| =  2^{2r+1}$ for some positive integer $r$ and
either $G \cong D_{8}^{*r}$ or $G \cong Q_8 * D_{8}^{*(r-1)}$.

(2) If $G$ is almost extraspecial, then $|G| = 2^{2r+2}$ for some positive integer $r$ and $G \cong D_{8}^{*r}*C_4$.

(3) If $G$ is generalized extraspecial, then either $G \cong E \times A$ or $G \cong (E*C_4)  \times A$, where $E$ is an extraspecial $2$-group and $A$ is an elementary abelian $2$-group.
\end{lem}

\begin{lem}\cite {TA}
Given a $2$-group $G$, we denote by $\mathcal{S}_{\alpha,\beta}(G)$ the set of all
elementary abelian sections $H_2/H_1\cong C_{2}^{\alpha}$ of $G$ with $|H_1|=2^{\beta}$.
 If $G$ is  (almost) extraspecial of order $2^n$, then
 \begin{center}
$|\mathcal{S}_{\alpha,\beta}(G)|\leq|\mathcal{S}_{\alpha,\beta}(D_8\times C_{2}^{n-3})|$, for all $\alpha$ and $\beta$.
\end{center}
\end{lem}

\begin{lem}
 If $G$ is an (almost) extraspecial group of order $2^n$, then
 \begin{center}
$\delta_2(G) \leq \delta_2(D_8\times C_{2}^{n-3}) $.
\end{center}
\end{lem}
\noindent {\bf Proof.} This follows from Lemma 2.7 taking $\alpha = 2$ and $\beta= 0$. $\Box$

\begin{lem} \cite {SU}
Let $A$ and $B$ be two finite groups. Then every subgroup $H$ of
the direct product $A\times B$ is completely determined by a quintuple $(A_1,A_2,B_1,B_2,\varphi)$
  where $A_1  \trianglelefteq  A_2 \leq A$, $B_1 \trianglelefteq B_2 \leq B$ and $\varphi : A_2/A_1 \longrightarrow B_2/B_1$ is an isomorphism, more exactly $H =\{(a, b) \in A_2\times  B_2 | \varphi(aA_1) = bB_1\}$. Moreover,
we have $|H| = |A_1||B_2| = |A_2||B_1|$.
\end{lem}

\section{ The Proof of Main Results}
\begin{prop}
Let $p$ be an odd prime and let $G$ be a group of order $p^n$ with  $\exp(G)=p^m$. Then
\begin{center}
$\delta(G)\leq \frac {p^{n-m+1}-p^n}{p^m-p^{m-1}}+s(G)-s_1(G)-1$,
\end{center}
\end{prop}

\noindent {\bf Proof.} Since $|G|=p^n$ and $\exp(G)=p^m$, we have
\begin{align*}
p^n&=1+(p-1)c_1(G)+(p^2-p)c_2(G)+\cdots+(p^m-p^{m-1})c_m(G)\\
&=1+(p-1)c_1(G)+(p^2-p)(s_2(G)-\delta_2(G))+\cdots+(p^m-p^{m-1})(s_m(G)-c_m(G))\\
&\leq 1+(p-1)c_1(G)+(p^m-p^{m-1})[(s_2(G)+\cdots+s_m(G))-(\delta_2(G)+\cdots\delta_m(G))].
\end{align*}
This implies that
\begin{center}
$\delta_2(G)+\cdots+\delta_m(G)\leq \frac{1-p^n}{p^m-p^{m-1}}+\frac{p-1}{p^m-p^{m-1}}c_1(G)+s_2(G)+\cdots+s_m(G)$.
\end{center}
Observe that $\exp(G)=p^m$, then $\delta_{k}(G)=s_{k}(G)$ for $k=m+1,\cdots, n$. It is obvious that $s_1(G)=c_1(G)$.
So we have
\begin{align*}
\delta(G)&=\delta_2(G)+\cdots+\delta_m(G)+\delta_{m+1}(G)+\cdots+\delta_{n}(G)\\
&\leq\frac{1-p^n}{p^m-p^{m-1}}+\frac{p-1}{p^m-p^{m-1}}s_1(G)+s_2(G)+\cdots+s_m(G)+\delta_{m+1}+\cdots+\delta_{n}(G)\\
&=\frac{1-p^n}{p^m-p^{m-1}}+\frac{p-1}{p^m-p^{m-1}}s_1(G)+(s_2(G)+\cdots+s_m(G)+s_{m+1}(G)+\cdots+s_n(G))\\
&=\frac{1-p^{n-m+1}+p^{n-m+1}-p^n}{p^m-p^{m-1}}+\frac{p-1}{p^m-p^{m-1}}s_1(G)+(s(G)-s_1(G)-1).
\end{align*}
Applying the condition $\exp(G)=p^m$ again, we can see that $\text{rank}(G)\leq n-m+1$. By Lemma 2.3(1), we have
\begin{center}
 $s_1(G)\leq \left [\begin{array} {c}
 n-m+1\\
 1
 \end{array}\right ]_p=\frac {p^{n-m+1}-1}{p-1}$.
 \end{center}
 This leads to   $1-p^{n-m+1}\leq (1-p)s_1(G)$. Consequently, we have
 \begin{align*}
 \delta(G)&\leq \frac{1-p^{n-m+1}+p^{n-m+1}-p^n}{p^m-p^{m-1}}+\frac{p-1}{p^m-p^{m-1}}s_1(G)+(s(G)-s_1(G)-1)\\
 &\leq \frac {p^{n-m+1}-p^n}{p^m-p^{m-1}}+s(G)-s_1(G)-1.
 \end{align*}
So the proof of Proposition is complete.
$\Box$

\vskip 0.3cm

\noindent {\bf Proof of Theorem 1.1.}  Suppose  $G$ is  a group of order $p^n$ with $\exp(G)=p^m$. Firstly, applying   Proposition 3.1,  we have
\begin{center}
$\delta(G) \leq \frac {p^{n-m+1}-p^n}{p^m-p^{m-1}}+s(G)-s_1(G)-1$.
\end{center}
Furthermore, observe that $\frac {p^{n-m+1}-p^n}{p^m-p^{m-1}}\leq 0$ and Lemma 2.3(2), we get
 \begin{center}
$\delta(G) \leq  s(G)-s_1(G)-1=\sum\limits_{k=2}^ns_k(G)\leq \sum\limits_{k=2}^n\left [\begin{array}{c}
                                                                              n \\
                                                                              k
                                                                            \end{array}\right ]_p$.
\end{center}
This implies that  equality holds  only if $m=1$.  Finally, applying  Lemma 2.3(2) again, we get that the equality holds if and only if $G$ is an elementary abelian $p$-group. The proof of the theorem is complete. $\Box$

\vskip 0.3cm

\noindent {\bf Proof of Theorem 1.2.}  Suppose  $G$ is  a group of order $p^n$ with $\exp(G)=p^m$. Firstly, applying   Proposition 3.1,  we have
\begin{center}
$\delta(G) \leq \frac {p^{n-m+1}-p^n}{p^m-p^{m-1}}+s(G)-s_1(G)-1$.
\end{center}
Furthermore, observe that $\frac {p^{n-m+1}-p^n}{p^m-p^{m-1}}\leq 0$ and Lemma 2.3(2), we get
 \begin{center}
$\delta(G) \leq  s(G)-s_1(G)-1=\sum\limits_{k=2}^ns_k(G)$.
\end{center}
This implies that the equality holds forces to $m=1$. By Lemma 2.4, we have $s_k(G) \leq s_k(\widetilde{G})$ for $\forall 1\leq k \leq  n$, where $\widetilde{G}= M_p(1, 1, 1) \times C_{p}^{n-3}$. So we get $\delta_c(G)\leq \delta_c(\widetilde{G})$. The proof of theorem is complete. $\Box$

 \begin{prop}
Assume $G$ is a group of order $p^n$ and $N$ is a normal subgroup of order $p$ of $G$,
then $\delta_k(G)\leq \delta_k(G/N\times C_p)$ for $1\leq k \leq n$.
\end{prop}
\noindent {\bf Proof.} Our proof will use induction on $|G|$. It is obvious if $|G|=p$ or $p^2$.
Let $\widetilde{G}=G/N\times\widetilde{N}$, where $|\widetilde{N}| = p$. By induction, we can assume that $\delta_k(H)\leq\delta_k(H/N\times \widetilde{N })$ for any subgroup $H$ of $G$ such that $N\leq H$.  Now, let
 \begin{center}
 $\mathcal{S}_1=\{H\leq G\mid N\leq H$, $|H|=p^k$ and $H$ is non-cyclic $\}$,

 $\mathcal{S}_2=\{H\leq G\mid N\not\subset H$, $|H|=p^k$ and $H$ is non-cyclic $\}$.
 \end{center}
 Then we have $\delta_k(G)=|\mathcal{S}_1 | +|  \mathcal{S}_2 |$. On the other hand, set
 \begin{center}
  $\mathcal{T}_1=\{\widetilde{H}\leq\widetilde{G}\mid \widetilde{N}\leq\widetilde{H}, |\widetilde{H}|=p^k\}$

  $\mathcal{T}_2=\{\widetilde{H}\leq\widetilde{G}\mid \widetilde{N}\not\leq\widetilde{H}, |\widetilde{H}|=p^k $ and $\widetilde{H}$ is non-cyclic$\}$.
  \end{center}
  Consequently,    $\delta_k (\widetilde{G})=| \mathcal{T}_1| +| \mathcal{T}_2 |$.
It is easy to see that,  for any subgroup $H$ of order $p^k$ of $G$ such that $N\leq H$, let $\widetilde{H}=H/N\times \widetilde{N}$,  then $\widetilde{H}\in  \mathcal{T}_1$  and hence  $| \mathcal{S}_1|\leq | \mathcal{T}_1|$.

  In the following,  it is sufficient for us to just prove that $\vert \mathcal{S}_2\vert\leq \vert \mathcal{T}_2\vert$.
Let  $\mathcal{S}$ denote the set of all non-cyclic subgroups of order $p^k$ of $G/N$  and $\mathcal{M}(H)$ (res. $\mathcal{M}(\widetilde{H})$) denote the set of all non-cyclic maximal subgroups of $H$ (res. $\widetilde{H}$) which do not contain $N$ (res. $\widetilde{N}$). Write $\mathcal{S}_3=\{H\leq G | H/N\in \mathcal{S}\}$
 and $\mathcal{T}_3=\{\widetilde{H}\leq \widetilde{G} | \widetilde{H}/\widetilde{N}\in \mathcal{S}\}$.
 We  claim $\mathcal{S}_2=\bigcup\limits_{H\in \mathcal{S}_3} \mathcal{M}(H)$ and $\mathcal{T}_2=\bigcup\limits_{\widetilde{H}\in \mathcal{T}_3} \mathcal{M}(\widetilde{H})$. We only show that $\mathcal{S}_2=\bigcup\limits_{H\in \mathcal{S}_3} \mathcal{M}(H)$,  and $\mathcal{T}_2=\bigcup\limits_{\widetilde{H}\in \mathcal{T}_3} \mathcal{M}(\widetilde{H})$ can be obtained by using a similar way.

Firstly, we show that $\mathcal{M}(H_1)\bigcap \mathcal{M}(H_2)=\varnothing$ if $H_1, H_2 \in \mathcal{S}_3$, where $H_1\neq H_2$. If not, assume $L \in\mathcal{M}(H_1)\bigcap\mathcal{M}(H_2)$. Then it is easy to see that $H_1 = LN =H_2$, a contradiction.

Secondly, we can prove that $\mathcal{S}_2=\bigcup\limits_{H\in \mathcal{S}_3}\mathcal{M}(H)$. It is obvious that $\bigcup\limits_{H\in \mathcal{S}_3} \mathcal{M}(H)\subseteq \mathcal{S}_2$. On the other hand, let $L\in \mathcal{S}_2$. Then $LN\in \mathcal{S}_3$ and $L \lessdot LN$. So $L \in \bigcup\limits_{H\in \mathcal{S}_3} \mathcal{M}(H)$. This implies that $\mathcal{S}_2\subseteq\bigcup\limits_{H\in \mathcal{S}_3} \mathcal{M}(H) $ and hence $\mathcal{S}_2=\bigcup\limits_{H\in \mathcal{S}_3}\mathcal{M}(H)$.

Finally, we need to prove that $  |\mathcal{M}(H) | \leq |\mathcal{M}(\widetilde{H})|$ for each $H\in \mathcal{S}_3$, where $\widetilde{H}=H/N\times \widetilde{N}$,    thus $|\mathcal{S}_2 |\leq | \mathcal{T}_2 |$. In fact, we know that
 \begin{center}
 $|\mathcal{M}(H)|= \delta_k(H)-(\delta_{k-1}(H/N)+|\nu_{k-1}(H/N)|)$  and
 $| \mathcal{M}(\widetilde{H})  | = \delta_k(\widetilde{H})-s_{k-1}(H/N)$,
 \end{center}
 where the set $\nu_{k-1}=\{L/N\leq H/N ||L|=p^k,L/N \text{ is cyclic, but }L\text { is non-cyclic}\}$.
 So we only show that $\delta_k(H)\leq  \delta_k(\widetilde{H})$.
 Observe that $H$ is a proper subgroup of $G$, we have assumed $\delta_k(H)\leq  \delta_k(\widetilde{H})$ by induction.  The proof of theorem is complete.
$\Box$
\vskip 0.3cm

Next we remark that Theorem 3.2 can be easily generalized in the following way.

\begin{thm}
  Let $G$ be a finite $p$-group of order $p^n$ and $M$ be a normal
subgroup of order $p^r$ of $G$. Then
\begin{center}
$\delta_k(G) \leq  \delta_k(G/M\times C_{p}^r)$, for any  $ 0 \leq  k \leq n$.
\end{center}
\end{thm}

\noindent {\bf Proof.} Take a chief series of $G$ containing $M$ and use induction on $r$. $\Box$

\begin{thm}
Let $G$ be a finite abelian $2$-group of order $2^n$ with  $n \geq 3$. If $G$ is
not elementary abelian, then
\begin{center}
$\delta(G) \leq \delta(D_8\times C_{2}^{n-3})$.
 \end{center}
\end{thm}

\noindent {\bf Proof.} Since $G$ is not elementary abelian, then
\begin{center}
$\delta_k(G) \leq \delta_k(C_4 \times  C_{2}^{n-2})$,  for all $0\leq k\leq n$,
\end{center}
 by using
Theorem 3.2 again and again. So it suffices to prove that
\begin{center}
$\delta_k(G_1\times  C_{2^{n-3}})\leq\delta_k(D_8\times C_{2^{n-3}})$, for all $2\leq k\leq n$.

\end{center}
where $G_1 = C_4\times C_2$. Observe that $\exp(D_8\times C_{2^{n-3}})=\exp(G_1\times  C_{2^{n-3}})=2^2$, and applying Lemma 2.5, if $3\leq k\leq n$, we have
\begin{center}
$\delta_k(G_1\times  C_{2^{n-3}})=s_k(G_1\times  C_{2^{n-3}}))\leq s_k(D_8\times C_{2^{n-3}})=\delta_k(D_8\times C_{2^{n-3}})$.
\end{center}
Thus, we only need to show that $\delta_2(G_1\times  C_{2^{n-3}})\leq \delta_2(D_8\times C_{2^{n-3}})$.

It is easy to calculate  that
 $\delta_2(G_1\times  C_{2^{n-3}})=\left [\begin{array} {c}
 n-1\\
 2
 \end{array}\right ]_2$,
  and
 \begin{center}
    $\delta_2(D_8\times C_{2^{n-3}})=
    2\cdot\left [\begin{array} {c}
 n-1\\
 2
 \end{array}\right ]_2-\left [\begin{array} {c}
 n-2\\
 2
 \end{array}\right ]_2$.
 \end{center}
So we have $\delta_2(G_1\times  C_{2^{n-3}})\leq \delta_2(D_8\times C_{2^{n-3}})$ and hence $\delta(G_1\times  C_{2^{n-3}})\leq \delta(D_8\times C_{2^{n-3}})$. The proof of theorem is complete. $\Box$
\vskip 0.3cm

\noindent {\bf Proof of Theorem 1.3.}  Suppose that $G$ is a non-elementary abelian group of order $2^n$ and $|G'| =2^m$. If $m = 0$, then $G$  is abelian and
the conclusion holds by Theorem 3.4. So we can assume that $m \geq 1$.

Firstly, suppose that $G/G'$ is not elementary abelian. Applying Theorem 3.3, we have
\begin{center}
$\delta_k(G)\leq \delta_k(G/G'\times C_{2}^{m})$, for any $2\leq k\leq n$.
\end{center}
Moreover, we obtain that
$\delta(G)\leq \delta(D_8\times C_{2^{n-3}})$  by Theorem 3.4.

Secondly, suppose that $G/G'$ is   elementary abelian, then $G'=\Phi(G)$.  Let $M$ be a normal
subgroup of $G$ such that $M \leq G'$ and $|G' :M|  = 2$.  Applying Theorem 3.3 again, we have
\begin{center}
$\delta_k(G)\leq\delta_k(G/M\times C_{2}^{m-1})$, for any $2\leq k\leq n$.
\end{center}
Write $G_1=G/M$, then $G_{1}'=\Phi(G_1)\leq Z(G_1)$ and $|G_{1}'|=2$. So $G_1$ is a generalized extraspecial $2$-group. Consequently, by Lemma 2.6, we have
\begin{center}
$G_1\cong E\times A$ or $(E*C_4)\times A$.
\end{center}
where $E$ is an extraspecial $2$-group and $A$ is an elementary abelian $2$-group. In other words, $G_1$ is a direct product of an (almost) extraspecial $2$-group and an elementary abelian $2$-group. So, it suffices to prove that if
$G_2$ is an (almost) extraspecial $2$-group of order $2^q$, then
\begin{center}
$\delta_k(G_2 \times C_{2}^{n-q})\leq\delta_k((D_8\times C_{2}^{q-3})\times C_{2}^{n-q})$.
\end{center}
 Observe that $\exp((D_8\times C_{2}^{q-3})\times C_{2}^{n-q})=\exp(G_2 \times C_{2}^{n-q})=2^2$, and applying Lemma 2.5,  we have
\begin{align*}
 \delta_k(G_2 \times C_{2}^{n-q})&=s_k(G_2 \times C_{2}^{n-q})\\
 &\leq s_k((D_8\times C_{2}^{q-3})\times C_{2}^{n-q})\\
 &=\delta_k((D_8\times C_{2}^{q-3})\times C_{2}^{n-q}),
\end{align*}
  for  $3\leq k\leq n$.
Thus, we only need to show that $\delta_2(G_2 \times C_{2}^{n-q})\leq \delta_2((D_8\times C_{2}^{q-3})\times C_{2}^{n-q})$.

In the following, let $A$ be one of the groups $G_2$ or $D_8\times C_{2}^{q-3}$.
  From Lemma 2.9,   it follows
that an elementary abelian subgroup $H$ of order $2^2$ of  $ A \times  C_{2}^{n-q}$
 is completely determined by
a quintuple $(A_1,A_2,B_1,B_2, \varphi)$, where $A_1 \trianglelefteq A_2 \leq A$, $B_1 \trianglelefteq B_2 \leq C_{2}^{n-q}$, $\varphi: A_2/A_1 \longrightarrow   B_2/B_1$ is an isomorphism, and $|A_2||B_1|=|A_1||B_2| = 2^2$.
Our proof will be divided into the    following three cases:

\vskip 0.3cm

Case I: $A_2 = 1$.

In this case,  we have $A_1=1$ and $B_1=B_2$ is one of the $\left [\begin{array} {c}
 n-q\\
 2
 \end{array}\right ]_2$
  subgroups of order $2^2$ of $C_{2}^{n-q}$. Clearly, these determine
  $\left [\begin{array} {c}
 n-q\\
 2
 \end{array}\right ]_2$
 distinct subgroups of $A\times C_{2}^{n-q}$ which are isomorphic to $C_{2}^2$.
\vskip 0.3cm
 Case II: $A_2\cong C_2$.

 Observe that  $|A_2||B_1|=2^2$,   $B_1$ is of order $2$  and can be chosen in$\left [\begin{array} {c}
 n-q\\
 1
 \end{array}\right ]_2$
ways. If $A_1 = A_2$
then $B_2 = B_1$, while if $A_1 = 1$ then $B_2$ is one of the $2^{n-q-1}-1$ subgroups
of order $2^2$ of $C_{2}^{n-q}$
 containing $B_1$. So  we obtain
 $s_1(A)\cdot2^{n-q-1}\cdot\left [\begin{array} {c}
 n-q\\
 1
 \end{array}\right ]_2$  distinct subgroups of $A\times C_{2}^{n-q}$ which are isomorphic to $C_{2}^2$.
\vskip 0.3cm

 Case III: $A_2\cong C_{2}^2$.

 In this case, $B_1=1$. If $A_1=A_2$, then $B_2=B_1=1$. If $|A_1| = 2$ then $B_2$ is one of the $2^{n-q}- 1$ subgroups
of order $2$ of $C_{2}^{n-q}$. If  $A_1 = 1$,  then $B_2$  is one of the
$\left [\begin{array} {c}
 n-q\\
 2
 \end{array}\right ]_2$  subgroups of order $2^2$ of $C_{2}^{n-q}$. So  we  obtain
 \begin{center}
 $\delta_2(A)\cdot\left ((2^{n-q}-1)+\left [\begin{array} {c}
 n-q\\
 2
 \end{array}\right ]_2\right)$
 \end{center}
  distinct subgroups of $A\times C_{2}^{n-q}$ which are isomorphic to $C_{2}^2$.
\vskip 0.3cm
 By above arguments, we get
 \begin{center}
 $ \delta_2(A\times C_{2}^{n-q})=\left [\begin{array} {c}
 n-q\\
 2
 \end{array}\right ]_2+s_1(A)\cdot2^{n-q-1}\cdot\left [\begin{array} {c}
 n-q\\
 1
 \end{array}\right ]_2+\delta_2(A)\cdot\left ((2^{n-q}-1)+\left [\begin{array} {c}
 n-q\\
 2
 \end{array}\right ]_2\right)$.
 \end{center}

Finally, by Lemmas 2.5 and 2.8, we know that
\begin{center}
$s_1(G_2)\leq s_1(D_8\times C_{2}^{q-3})$ and $\delta_2(G_2)\leq\delta_2(D_8\times C_{2}^{q-3})$.
\end{center}
Consequently, we have
\begin{align*}
\delta_2(G_2\times C_{2}^{n-q})&=\left [\begin{array} {c}
 n-q\\
 2
 \end{array}\right ]_2+s_1(G_2)\cdot2^{n-q-1}\cdot\left [\begin{array} {c}
 n-q\\
 1
 \end{array}\right ]_2+\delta_2(G_2)\cdot\left ((2^{n-q}-1)+\left [\begin{array} {c}
 n-q\\
 2
 \end{array}\right ]_2\right) \\
 &\leq \left [\begin{array} {c}
 n-q\\
 2
 \end{array}\right ]_2+s_1(D_8\times C_{2}^{q-3})\cdot2^{n-q-1}\cdot\left [\begin{array} {c}
 n-q\\
 1
 \end{array}\right ]_2\\
 &+\delta_2(D_8\times C_{2}^{q-3})\cdot\left ((2^{n-q}-1)+\left [\begin{array} {c}
 n-q\\
 2
 \end{array}\right ]_2\right)\\
 &= \delta_2((D_8\times C_{q-3})\times C_{2}^{n-q}).
\end{align*}
Now we have shown that $\delta(G)\leq \delta(D_8\times C_{2}^{n-3})$. The proof of theorem is complete.
$\Box$\\

\noindent{\bf Availability of data and materials}

Data sharing not applicable to this article as no data sets were generated or
analysed during the current study. \\

\noindent{\bf Declarations
Competing interests }

The authors have no relevant financial or non-financial interests to disclose. \\



\bibliographystyle{amsplain}

\end{document}